\documentclass{amsart}

\usepackage{amsmath}
\usepackage{amsfonts}
\usepackage{amssymb,enumerate}
\usepackage{amsthm}
\usepackage[all]{xy}

\theoremstyle{plain}
\newtheorem{lem}{Lemma}[section]
\newtheorem{cor}[lem]{Corollary}

\newtheorem{thm}[lem]{Theorem}

\theoremstyle{definition}
\newtheorem{defn}[lem]{Definition}
\newtheorem{ex}[lem]{Example}
\newtheorem{question}[lem]{Question}

\newtheorem{disc}[lem]{Remark}

\newtheorem{fact}[lem]{Fact}

\newtheorem{remarkdefinition}[lem]{Remark/Definition}

\newcommand{\pd}{\operatorname{pd}}

\newcommand{\rank}{\operatorname{rank}}	

\newcommand{\edim}{\operatorname{edim}}

\newcommand{\len}{\operatorname{len}}

\newcommand{\ext}{\operatorname{Ext}}

\newcommand{\Hom}{\operatorname{Hom}}	

\newcommand{\spec}{\operatorname{Spec}}
\newcommand{\s}{\mathfrak{S}}

\newcommand{\im}{\operatorname{Im}}

\newcommand{\Ker}{\operatorname{Ker}}

\newcommand{\ideal}[1]{\mathfrak{#1}}
\newcommand{\m}{\ideal{m}}
\newcommand{\n}{\ideal{n}}
\newcommand{\p}{\ideal{p}}

\newcommand{\fr}{\ideal{r}}

\newcommand{\comp}[1]{\widehat{#1}}

\newcommand{\wti}{\widetilde}

\newcommand{\ass}{\operatorname{Ass}}

\newcommand{\Min}{\operatorname{Min}}

\newcommand{\bbq}{\mathbb{Q}}

\newcommand{\xra}{\xrightarrow}

\newcommand{\vf}{\varphi}

\newcommand{\y}{\mathbf{y}}

\newcommand{\x}{\mathbf{x}}

\renewcommand{\geq}{\geqslant}
\renewcommand{\leq}{\leqslant}
\renewcommand{\ker}{\Ker}
\renewcommand{\hom}{\Hom}

\newcommand{\Ext}[4][R]{\operatorname{Ext}_{#1}^{#2}(#3,#4)}

\newcommand{\Otimes}[3][R]{#2\otimes_{#1}#3}
\renewcommand{\Hom}[3][R]{\operatorname{Hom}_{#1}(#2,#3)}	
\newcommand{\Tor}[4][R]{\operatorname{Tor}^{#1}_{#2}(#3,#4)}

\newcommand{\cd}{C^{\dagger}}

\numberwithin{equation}{lem}

\begin{document}

\bibliographystyle{amsplain}

\author{Susan M.\ Cooper}

\address{Susan M.\ Cooper,
Department of Mathematics, Pearce Hall 214, Central Michigan University, Mount Pleasant, MI 48859
USA}

\email{s.cooper@cmich.edu}

\urladdr{http://www.math.unl.edu/\~{}scooper4/}

\author{Sean Sather-Wagstaff}

\address{Sean Sather-Wagstaff, Department of Mathematics, 
North Dakota State University, Dept. \# 2750,
PO Box 6050,
Fargo, ND 58108-6050
USA}

\email{Sean.Sather-Wagstaff@ndsu.edu}

\urladdr{http://www.ndsu.edu/pubweb/\~{}ssatherw/}

\thanks{Sean Sather-Wagstaff was supported in part by a grant from the NSA}

\title{Multiplicities  of semidualizing modules}

\date{\today}



\begin{abstract}
A finitely generated module $C$
over a commutative noetherian ring $R$ is semidualizing
if $\Hom CC\cong R$ and $\Ext iCC =0$ for all $i\geq 1$.
For certain local Cohen-Macaulay rings $(R,\m)$, we verify the equality 
of Hilbert-Samuel multiplicities $e_R(J;C)=e_R(J;R)$
for all semidualizing $R$-modules $C$ and all $\m$-primary ideals $J$.
The classes of rings we investigate include 
those that are determined
by ideals defining fat point schemes in projective space or 
by monomial ideals.
\end{abstract}

\maketitle

\section{Introduction} \label{sec1}

In this section, let $(R,\m,k)$ be a Cohen-Macaulay local ring with a dualizing module $D$.
A finitely generated $R$-module $C$ is \emph{semidualizing}
if $\Hom CC\cong R$ and $\Ext iCC =0$ for all $i\geq 1$.
Thus, the module $D$ is precisely a semidualizing module of finite injective dimension.
Let $\s_0(R)$ denote the set of isomorphism classes of semidualizing $R$-modules.
(See Section~\ref{sec2} for definitions and background information.)
For example, the $R$-modules $R$ and $D$ are semidualizing. 
The ring $R$ is Gorenstein if and only if $D\cong R$, equivalently,
if and only if $\s_0(R)=\{[R]\}$.

In this paper, we investigate the following question, motivated by the well-known 
equality $e_R(J;D)=e_R(J;R)$:

\begin{question} \label{q101}
Let $C$ be a semidualzing $R$-module.
For each $\m$-primary ideal $J$, must we have an equality of Hilbert-Samuel
multiplicities $e_R(J;C)=e_R(J;R)$?
\end{question}

When $R$ is generically Gorenstein (e.g., reduced)
an affirmative answer to this question is contained in~\cite[(2.8(a))]{sather:divisor}. 
In Theorems~\ref{thm201} and~\ref{lem204} we address a few more  cases with the following:

\begin{thm} \label{thm102}
Assume that $R$ 
satisfies one of the following conditions:
\begin{enumerate}[\quad\rm(1)]
\item  \label{thm102b}
$P^2R_P=0$
for each $P\in\ass(R)$;
\item  \label{thm102c}
$\comp R\cong k[\![X_0,X_1,\cdots,X_n]\!]/Ik[\![X_0,X_1,\cdots,X_n]\!]$ 
where $I\subseteq k[X_0,X_1,\cdots,X_n]$ is the ideal
determining
a fat point scheme in $\mathbb{P}^n_k$; or
\item  \label{thm102d}
$\comp R\cong k[\![X_1,\ldots,X_n]\!]/I$ 
where $I$
is generated by monomials in the $X_i$.
\end{enumerate}
For every $\m$-primary ideal
$J\subset R$ and every semidualizing $R$-module $C$, we have $e_{R}(J;C)=e_{R}(J;R)$.
\end{thm}

This paper is organized as follows.
Section~\ref{sec2} consists of background material,
and Section~\ref{sec5} contains the proof of Theorem~\ref{thm102}.

\section{Background} \label{sec2}

For the rest of this paper, let $R$ and $S$ be  commutative noetherian rings of finite Krull dimension.

\begin{defn} \label{defn201}
Let  $C$ be an  $R$-module.  The natural \emph{homothety map}
$$\chi^R_C\colon R\to \hom CC$$
is the $R$-module homomorphism  given by 
$\chi^R_C(r)(c)=rc$.
The module $C$ is \emph{semidualzing} 
if it satisfies the following:
\begin{enumerate}[\quad(1)]
\item  \label{defn202a}
$C$ is finitely generated;
\item  \label{defn202b}
the homothety map
$\chi^R_C\colon R\to \hom CC$ is an isomorphism; and
\item  \label{defn202c}
$\ext^i_R(C,C)=0$ for all $i>0$.
\end{enumerate}
The module $C$ is \emph{dualizing} if it is 
semidualizing and has finite injective dimension.\footnote{The assumption 
$\dim(R)<\infty$ guarantees that a finitely generated
$R$-module 
$C$ has finite injective dimension over $R$ if and only if 
$C_{\m}$ has finite injective dimension over $R_{\m}$.
For instance, this removes the need to worry about any distinction
between the terms ``dualizing'' and ``locally dualizing'', and similarly for
``Gorenstein'' and ``locally Gorenstein''.  This causes no loss of generality
in this paper as we are primarily concerned with local  and graded situations.}
\end{defn}

\begin{ex} \label{ex201}
It is straightforward to show that the free $R$-module $R^1$ is semidualizing.
It is  dualizing if and only if $R$ is  Gorenstein.
\end{ex}

The following facts will be used in the sequel.

\begin{fact} \label{fact201''}
Let $C$ be a semidualizing $R$-module. Then a sequence
$x_1,\ldots,x_n\in R$ is $C$-regular if and only if it is $R$-regular.
(See, e.g., \cite[(1.4)]{sather:bnsc} for a brief explanation of the local case.
The general case has the same proof.)
\end{fact}

\begin{fact} \label{fact201}
Assume that $R$ is Cohen-Macaulay and that $D$ is a dualizing $R$-module.
Let $C$ be a semidualizing $R$-module. From~\cite[(3.1), (3.4)]{frankild:rrhffd}
and~\cite[(V.2.1)]{hartshorne:rad}, we have the following:
\begin{enumerate}[\quad(a)]
\item \label{fact201a}
$\Ext{i}{C}{D}=0$ for all $i\geq 1$;
\item \label{fact201b}
the dual
$\Hom{C}{D}$ is a semidualizing $R$-module;
\item \label{fact201e}
the natural biduality map
$\delta^D_C\colon C\to \Hom{\Hom{C}{D}}{D}$
given by the formula
$\delta^D_C(c)(\psi)=\psi(c)$
is an isomorphism;
\item \label{fact201c}
$\Tor{i}{C}{\Hom{C}{D}}=0$ for all $i\geq 1$; and
\item \label{fact201d}
the natural evaluation map
$\Otimes{C}{\Hom{C}{D}}\to D$ given by
$\Otimes[]{c}{\psi}\mapsto\psi(c)$ is an isomorphism.
\end{enumerate}
From~\eqref{fact201e}, we conclude that:
\begin{enumerate}[\quad(f)]
\item \label{fact201f}
if $\Hom{C}{D}\cong R$,
then $C\cong D$.
\end{enumerate}
Assume that $R$ is local.
Because of~\eqref{fact201c} and~\eqref{fact201d}, the minimal free resolution of $D$
is obtained by tensoring the minimal free resolutions
of $C$ and $\Hom{C}{D}$. In particular, this implies that:
\begin{enumerate}[\quad(g)]
\item \label{fact201g}
$\beta^R_i(D)=\sum_{j=0}^i\beta^R_j(C)\beta^R_{i-j}(\Hom CD)$ for each $i\geq 0$.
\end{enumerate}
\end{fact}

\begin{fact} \label{fact205}
Let $\vf\colon R\to S$ be a homomorphism of commutative noetherian rings.
Assume that $S$ has finite flat dimension as an $R$-module.
For example, this is satisfied when $S$ is flat as an $R$-module,
or when $\vf$ is surjective with $\Ker(\vf)$ generated by an $R$-regular sequence.
If $C$ is a semidualizing $R$-module, then $\Otimes SC$ is a 
semidualizing $S$-module; the converse holds when 
$\vf$ is faithfully flat; see~\cite[(4.5)]{frankild:rrhffd}.
Thus, the rule of assignment $[C]\mapsto[S\otimes_RC]$ describes a well-defined
function $\s_0(\vf)\colon\s_0(R)\to\s_0(S)$. If the map $\vf$ is local,
that is if $(R,\m)$ and $(S,\n)$ are local and $\vf(\m)\subseteq\n$,
then the induced map $\s_0(\vf)$ is injective;
see~\cite[(4.9)]{frankild:rrhffd}. 

Assume that $\vf$ is local and 
satisfies one of the following conditions:
\begin{enumerate}[\quad(1)]
\item\label{fact205a}
$\vf$ is flat with Gorenstein closed fibre $S/\m S$
(e.g., $\vf$ is the natural map from $R$ to its completion $\comp R$); or
\item\label{fact205b}
$\vf$ is surjective with $\Ker(\vf)$ generated by an $R$-regular sequence.
\end{enumerate}
Then a semidualizing $R$-module $C$ is dualizing for $R$ if and only if
$S\otimes_RC$ is dualizing for $S$ by~\cite[(3.1.15),(3.3.14)]{bruns:cmr}.
When $R$ is complete and $\vf$ satisfies condition~\eqref{fact205b},
the induced map $\s_0(\vf)\colon\s_0(R)\to\s_0(S)$ is bijective;
see~\cite[(4.2)]{frankild:sdcms} or~\cite[(2)]{gerko:osmagi}.
\end{fact}

\begin{fact} \label{fact202}
Assume that $(R,\m,k)$ is local and $C$ is a semidualizing $R$-module.
If $C$ has finite projective dimension, then $C\cong R$; see, e.g., \cite[(1.14)]{sather:bnsc}.
If $R$ is Gorenstein, then $C\cong R$ by~\cite[(8.6)]{christensen:scatac}.
If $\m^2=0$,
then either $C\cong R$ or $C$ is dualizing for $R$.
(Indeed, if $C\ncong R$, then the first syzygy $C'$ of $C$ is a non-zero $k$-vector space
such that $\ext^1_R(C',C)=0$, so $C$ is injective.)
\end{fact}

The following notions are standard.

\begin{remarkdefinition} \label{disc202}
Let $(R,\m)$ be a local ring and let $I$ be an $\m$-primary ideal of $R$.
Let $C$ be a finitely generated $R$-module of dimension $d$.
There is a polynomial $H_{I,C}(j)\in\bbq[j]$ such that
$H_{I,C}(j)=\len_R(I^jC/I^{j+1}C)$ for $j\gg 0$.
Moreover, the degree of $H_{I,C}(j)$ is $d-1$,
and the leading coefficient is of the form
$e_R(I;C)/(d-1)$ for some positive integer 
$e_R(I;C)$. The integer $e_R(I;C)$ is the
\emph{Hilbert-Samuel} multiplicity of $C$ with respect to $I$.
\end{remarkdefinition}

The next lemma is a version of a result of Herzog~\cite[(2.3)]{herzog:mlr}.
It is proved similarly and is almost certainly well-known.

\begin{lem} \label{lem205}
Let $\vf\colon(R,\m)\to(S,\n)$ be a flat local ring homomorphism
such that $\m S=\n$.
Let $I$ be an $\m$-primary ideal of $R$,
and let $C$ be a finitely generated $R$-module.
Set $\wti C=S\otimes_RC$ and $\wti I=IS$.
For each $j$ there is an equality 
$\len_{S}(\wti I^j\wti C/\wti I^{j+1}\wti C)
=\len_R(I^jC/I^{j+1}C)$.
In particular, we have
$e_{S}(\wti I; \wti C)=e_R(I;C)$.
\end{lem}

We end this section with a  discussion of fat point schemes.

\begin{defn} \label{defn299}
Let $k$ be a field. Fix distinct points $Q_1,\ldots,Q_r\in\mathbb{P}^n_k$
and integers $m_1,\ldots,m_r\geq 1$. 
Set $S_0=k[X_0,X_1,\ldots,X_n]$ with 
irrelevant maximal ideal $\n_0=(X_0,X_1,\ldots,X_n)S_0$.
For each index $j$, let $I(Q_j)\subset S_0$ be the (reduced) vanishing ideal of $Q_j$.
The subscheme of $\mathbb{P}^n_k$ 
defined by the ideal $I=\cap_{j=1}^rI(Q_j)^{m_j}\subseteq S_0$
is the \emph{fat point scheme} determined by the points
$Q_1,\ldots,Q_r$
with multiplicities $m_1,\ldots,m_r$.
\end{defn}

\begin{disc} \label{disc299}
Continue with the notation of Definition~\ref{defn299}.

Set $S=k[\![X_0,X_1,\cdots,X_n]\!]$ with maximal ideal $\n=(X_0,X_1,\cdots,X_n)S$.
The local rings $(S_0)_{\n_0}/I_{\n_0}$ and $R=S/IS$ are Cohen-Macaulay of dimension 1.

Note that the quotient
$S_0/I(Q_j)$ is  isomorphic (as a graded $k$-algebra) to a polynomial ring $k[Y]$.
In particular, the completion of the local ring $(S_0)_{\n_0}/I(Q_j)(S_0)_{\n_0}$
(isomorphic to $S/I(Q_j)S$)
is isomorphic to the formal power series ring $k[\![Y]\!]$. 
In particular, the ideal $I(Q_j)S$ is prime.
It follows that the associated primes of
$R=S/I$ are of the form $P_j=I(Q_j)S/I$.
Localizing at one of these primes yields
$$R_{P_j}\cong S_{I(Q_j)}/IS_{I(Q_j)}
\cong S_{I(Q_j)}/I(Q_j)^{m_j}S_{I(Q_j)}
\cong S_{I(Q_j)}/(I(Q_j)S_{I(Q_j)})^{m_j}.
$$
In other words, we have $R_{P_j}\cong S_j/\n_j^{m_j}$ 
for some regular local ring $(S_j,\n_j)$.
\end{disc}

\section{Multiplicities of Semidualizing Modules} \label{sec5}

In this section, we consider 
Question~\ref{q101}  for 
certain classes of rings.

\begin{lem} \label{lem202}
Let $(S,\n)$ be a regular local ring containing a field. Let $e$ be a positive integer,
and set $R=S/\n^e$.
Let $C$ be a  semidualizing $R$-module.
Then either $C\cong R$ or $C$ is dualizing for $R$.
In particular, we have $\len_{R}(C)=\len_{R}(R)$.
\end{lem}

\begin{proof}
Fact~\ref{fact202} deals with the case $e=1$, so assume that $e\geq 2$.
The ring $R$ is artinian and local. Hence, it is complete and has a dualizing module $D$. 
There are isomorphisms
$$R\cong \comp{R}\cong\comp{S}/\n^e\comp{S}\cong
k[\![X_1,\ldots,X_n]\!]/(X_1,\ldots,X_n)^e$$
where $k$ is a field and $n=\edim(S)$.
We now conclude  from~\cite[(4.11)]{sather:divisor}
that $C\cong R$ or $C\cong D$.
The conclusion $\len_{R}(C)=\len_{R}(R)$ follows from the well-known equality
$\len_{R}(D)=\len_{R}(R)$.
\end{proof}

The next result contains 
cases~\eqref{thm102b} and~\eqref{thm102c} of
Theorem~\ref{thm102} from the introduction.

\begin{thm} \label{thm201}
Let $(R,\m)$ be a local Cohen-Macaulay ring, and let $C$ be a semidualizing
$R$-module. Assume that $R$
satisfies one of the following conditions:
\begin{enumerate}[\quad\rm(1)]
\item  \label{thm201b}
$P^2R_P=0$
for each $P\in\ass(R)$;
\item  \label{thm201c}
$\comp R\cong k[\![X_0,X_1,\cdots,X_n]\!]/Ik[\![X_0,X_1,\cdots,X_n]\!]$ 
where $I\subseteq k[X_0,X_1,\cdots,X_n]$ is the ideal
determining
a fat point scheme in $\mathbb{P}^n_k$.
\end{enumerate}
Then for every $\m$-primary ideal
$J\subset R$, we have $e_{R}(J;C)=e_{R}(J;R)$.
\end{thm}

\begin{proof}
\eqref{thm201b}
Assume that $P^2R_P=0$
for each $P\in\ass(R)$. Fact~\ref{fact202} implies that 
$\len_{R_P}(C_P)=\len_{R_P}(R_P)$ 
for each $P\in\ass(R)$,
hence the second equality in the following sequence
wherein each sum is taken over all $P\in\ass(R)$:
\begin{align*}
e(J;C)
&=\sum_{P}\len_{R_P}(C_P)e(J;R/P)
=\sum_{P}\len_{R_P}(R_P)e(J;R/P)
=e(J;R).
\end{align*}
The remaining equalities follow from the additivity formula~\cite[(4.7.t)]{bruns:cmr}.

\eqref{thm201c}
Using Fact~\ref{fact205} and Lemma~\ref{lem205} we may pass to the completion 
to assume that $R\cong\comp R$.
For each $P\in\ass(R)$, Remark~\ref{disc299} implies that $R_{P}\cong S/\n^{m}$ 
for some regular local ring $(S,\n)$.
Lemma~\ref{lem202} implies that 
$\len_{R_P}(C_P)=\len_{R_P}(R_P)$ 
for each $P\in\ass(R)$, hence the desired conclusion follows
as in case~\eqref{thm201b}.
\end{proof}

The next result contains part of case~\eqref{thm102d} of Theorem~\ref{thm102} from the introduction.
The general case is in Theorem~\ref{lem204}.

\begin{lem} \label{lem204z}
Let $(A,\fr)$ be a complete reduced local ring,
and set $S=A[\![x_1,\ldots,x_n]\!]$, the formal power series ring,
with maximal ideal $\n=(\fr,x_1,\ldots,x_n)S$.
Let $I\subset S$ be an ideal generated by monomials in the $x_i$,
and set $R=S/I$ with maximal ideal $\m=\n/I$.
Assume that $R$ is Cohen-Macaulay, and let $C$ be a  semidualizing $R$-module.
Then for each  $P\in\spec(R)$ 
and for each $PR_P$-primary ideal
$J\subset R_P$, we have $e_{R_P}(J;C_P)=e_{R_P}(J;R_P)$.
\end{lem}

\begin{proof}
Here is an outline of the proof.
We show that the theory of polarization for monomial ideals yields
a complete reduced Cohen-Macaulay local ring $R^*$ and a surjection
$\tau\colon R^*\to R$ such that $\ker(\tau)$ is generated by
an $R^*$-regular sequence $\y$. Facts~\ref{fact201''} and~\ref{fact205}
yield a semidualizing $R^*$-module such that the sequence $\y$ is
$C^*$-regular and $C^*/(\y)C^*\cong C$. Because $R^*$ is complete and reduced,
the desired conclusion follows from~\cite[(2.8.b)]{sather:divisor}.

Set $S_0=A[x_1,\ldots,x_n]\subset S$.
The ideal $I_0=I\cap S_0$ is generated by monomials in the $x_i$, in fact,
by the same list of monomial generators used to generate $I$.

The theory of polarization for monomial ideals yields the following:
\begin{enumerate}[\quad(1)]
\item
a polynomial ring $S_0^*=A[x_{1,1},\ldots, x_{1,t_1},x_{2,1},\ldots, x_{2,t_2},\ldots,x_{n,1},\ldots, x_{n,t_n}]$
with irrelevant maximal ideal
$\n_0^*=
(\fr,\{x_{i,j}\})S_0^*$,
\item
an ideal $I_0^*\subseteq S_0^*$ generated by \emph{square-free} monomials in
the $x_{i,j}$,
\item a sequence $\y=y_1,\ldots,y_r\in\n_0^*$
that is both $S_0^*$-regular and $(S_0^*/I_0^*)$-regular
and such that $S_0^*/(\y)S_0^*\cong S_0$ and
$S_0^*/(I_0^*+(\y)S_0^*)\cong S_0/I_0$.
\end{enumerate}
Localizing at $\n_0^*$ and passing to the completion yields the following:
\begin{enumerate}[\quad($1'$)]
\item
a power series ring $S^*=A[\![x_{1,1},\ldots, x_{1,t_1},x_{2,1},\ldots, x_{2,t_2},\ldots,x_{n,1},\ldots, x_{n,t_n}]\!]$
over $k$ with maximal ideal denoted
$\n^*=
(\fr,\{x_{i,j}\})S^*$,
\item
an ideal $I^*=I_0^*S^*\subseteq S^*$ generated by square-free monomials in
the $x_{i,j}$,
\item a sequence $\y=y_1,\ldots,y_r\in\n^*$
that is both $S^*$-regular and $(S^*/I^*)$-regular
and such that $S^*/(\y)S^*\cong S$ and
$S^*/(I^*+(\y)S^*)\cong S/I=R$.
\end{enumerate}
Setting $R^*=S^*/I^*$, we have the following:
\begin{enumerate}[\quad($1''$)]
\item
Since $I^*$ is generated by square-free monomials, the ring $R^*$ is reduced.
\item 
The sequence $\y$ is $R^*$-regular such that $R^*/(\y)R^*\cong R$.
In particular, since $R$ is Cohen-Macaulay, so is $R^*$.
\item
Since $S^*$ is complete, so is $R^*$. 
Thus, Fact~\ref{fact205} 
provides a semidualizing $R^*$-module $C^*$ such that $C\cong C^*\otimes_{R^*}R$.
\item
Since the sequence $\y$ is $R^*$-regular, it is also $C^*$-regular
by Fact~\ref{fact201''}.
\end{enumerate}
Let $\tau\colon R^*\to R$ be the canonical surjection, and set $P^*=\tau^{-1}(P)$.
We then have the following:
\begin{enumerate}[\quad($1'''$)]
\item 
Since $R^*$ is reduced, so is the localization $R^*_{P^*}$.
\item 
Since $R^*$ is Cohen-Macaulay, so is the localization $R^*_{P^*}$.
In particular, the ring $R^*_{P^*}$ is equidimensional.
Also, we have $R^*_{P^*}/(\y)R^*_{P^*}\cong R_P$.
\item
Since $R^*$ is complete, it is excellent, and it follows that 
the localization $R^*_{P^*}$ is also excellent. In particular, for every
$\p\in\Min(R^*_{P^*})$ the ring $(R^*_{P^*})_{\p}/\p(R^*_{P^*})_{\p}\otimes_{R^*_{P^*}}\widehat{R^*_{P^*}}$
is Gorenstein.
\item The $R^*_{P^*}$-module $C^*_{P^*}$ is semidualizing
and satisfies $C^*_{P^*}/(\y)C^*_{P^*}\cong C_P$.
\end{enumerate}
Using the conditions ($1'''$)---($4'''$),
the conclusion $e_{R_P}(J;C_P)=e_{R_P}(J;R_P)$ 
now follows from~\cite[(2.8.b)]{sather:divisor}.
\end{proof}

The next result contains case~\eqref{thm102d} of Theorem~\ref{thm102} from the introduction.
In preparation, recall that
a prime ideal $P$ in a local ring $R$
is \emph{analytically unramified} if the completion
$\widehat{R/P}$ is reduced. For example, if $R$ is excellent,
then every prime ideal of $R$ is analytically unramified.

\begin{thm} \label{lem204}
Let $(A,\fr)$ be a complete reduced local ring,
and $S=A[\![x_1,\ldots,x_n]\!]$ the formal power series ring,
with maximal ideal $\n=(\fr,x_1,\ldots,x_n)S$.
Let $I\subset S$ be an ideal generated by monomials in the $x_i$.
Let $R$ be a local Cohen-Macaulay ring
such that $\comp R\cong S/I$, and let $C$ be a  semidualizing $R$-module.
\begin{enumerate}[\quad\rm(a)]
\item \label{lem204a}
Let  $P\in\spec(R)$ be analytically unramified.
Then for every $PR_P$-primary ideal
$J\subset R_P$, we have $e_{R_P}(J;C_P)=e_{R_P}(J;R_P)$.
\item \label{lem204c}
For every $\m$-primary ideal
$J\subset R$, we have $e_{R}(J;C)=e_{R}(J;R)$.
\end{enumerate}
\end{thm}

\begin{proof}
\eqref{lem204a}
Since the natural map $R\to\comp R$ is flat and local, 
there is a prime $\wti{P}\in\spec(\comp{R})$ such that $P=\wti{P}\cap R$
and that the induced map $R_P\to \comp{R}_{\wti{P}}$ is flat and local.

The $R_P$-module $C_P$ is semidualizing.
Furthermore, by flat base-change, the $\comp{R}_{\wti{P}}$-module
$\comp{R}_{\wti{P}}\otimes_{R_P}C_P$ is semidualizing.
The fact that $P$ is analytically unramified implies that
the maximal ideal of $R_P$ extends to the maximal ideal of $\comp{R}_{\wti{P}}$.
Hence, Lemma~\ref{lem205} yields the first and third equalities in the following sequence:
\begin{align*}
e_{R_P}^{}(J;C_P)
=e^{}_{\comp{R}_{\wti{P}}}(J\comp{R}_{\wti{P}};\comp{R}_{\wti{P}}\otimes_{R_P}C_P)
=e^{}_{\comp{R}_{\wti{P}}}(J\comp{R}_{\wti{P}};\comp{R}_{\wti{P}})
=e_{R_P}^{}(J;R_P).
\end{align*}
The second equality is from Lemma~\ref{lem204z}.

\eqref{lem204c}
Since $R/\m$ is a field, it is complete.
Hence, the prime ideal $\m$ is analytically unramified,
so the desired conclusion follows from part~\eqref{lem204a}.
\end{proof}

\begin{cor} \label{lem204y}
Let $(S,\n)$ be a regular local ring containing a field, 
and let $\x=x_1,\ldots,x_n$ be a regular system of parameters for $S$.
Let $I\subset S$ be an ideal generated by monomials in the $x_i$,
and set $R=S/I$ with maximal ideal $\m=\n/I$.
Assume that $R$ is Cohen-Macaulay, and let $C$ be a  semidualizing $R$-module.
For every $P\in\ass(R)$, we have $\len_{R_P}(C_P)=\len_{R_P}(R_P)$.
\end{cor}

\begin{proof}
Since $R$ is Cohen-Macaulay, we have $P\in\Min(R)$.
This explains the first and third equalities in the next sequence:
$$\len_{R_P}(C_P)
=e_{R_P}^{}(PR_P;C_P)
=e_{R_P}^{}(PR_P;R_P)
=\len_{R_P}(R_P).$$
For the second equality, it suffices to show that $P$ is analytically unramified;
then the equality follows from Theorem~\ref{lem204}\eqref{lem204a}.

Since $I$ is generated by monomials in the $x_i$, the associated prime $P$ 
has the form
$P=(x_{i_1},\ldots,x_{i_j})R$. 
This is, of course, standard when $S$ is a polynomial ring.
Since $S$ is not a polynomial ring, we justify this statement.
First note that each ideal 
$(x_{i_1},\ldots,x_{i_j})R$ is prime because the sequence
$\x$ is a regular system of parameters.
Since $R=S/I$ is Cohen-Macaulay, the prime $P$ is minimal in $\spec(R)$.
Let $\pi\colon S\to R$ be the canonical surjection, and set $Q=\pi^{-1}(P)$.
The prime $Q$ is a minimal prime for any primary decomposition of $I$,
and it follows that $Q$ is a minimal prime for any primary decomposition of the
radical $\sqrt{I}$.

Because the sequence $\x$ is regular and contained in the Jacobson radical of $S$,
a result of Heinzer, Mirbagheri, Ratliff, and Shah~\cite[(4.10)]{heinzer:pdmiii}
implies that there are non-negative integers
$u,e_{1,1},\ldots,e_{1,n},e_{2,1},\ldots,e_{2,n},\ldots,e_{u,1},\ldots,e_{u,n}$ such that
$$I=\cap_{s=1}^{u}(x_1^{e_{s,1}},\ldots,x_n^{e_{s,n}})S.$$
Since each ideal 
$(x_{i_1},\ldots,x_{i_j})S$ is prime,
it is straightforward to show that one has
$\sqrt{(x_1^{e_{s,1}},\ldots,x_n^{e_{s,n}})S}
=(x_{i_1},\ldots,x_{i_j})S$
and hence
\begin{equation} 
\label{lem204e}
\sqrt{I}=\cap_{s=1}^{u}(x_1^{\epsilon_{s,1}},\ldots,x_n^{\epsilon_{s,n}})S
\end{equation}
where 
$$\epsilon_{s,i}=\begin{cases}
0 & \text{if $e_{s,i}=0$} \\
1 & \text{if $e_{s,i}\neq 0$.} 
\end{cases}
$$
Since each $(x_1^{\epsilon_{s,1}},\ldots,x_n^{\epsilon_{s,n}})S$ is prime,
the intersection~\eqref{lem204e} is a primary decomposition.
It follows that $P=(x_1^{\epsilon_{s,1}},\ldots,x_n^{\epsilon_{s,n}})S$
for some index $s$, so $P$ has the desired form.

It follows that 
$R/P\cong S/(x_{i_1},\ldots,x_{i_j})S$ is a regular local ring.
Thus, the completion $\widehat{R/P}$ is also a regular local ring.
In particular, the ring $\widehat{R/P}$ is an integral domain, so it is reduced,
and $P$ is analytically unramified by definition.
\end{proof}


We conclude with three results relating lengths and multiplicities to
Betti numbers of semidualizing modules, starting with a general result
for modules of infinite projective dimension.

\begin{lem} \label{lem201}
Let 
$R$ be a local ring such that $\ass(R)=\Min(R)$. Let 
$C$ be a finitely generated $R$-module of infinite projective dimension,
and consider an exact sequence
$$R^{a_1}\xra{\partial} R^{a_0}\to C\to 0.$$
Assume that for each $P\in\ass(R)$ 
one has $\len_{R_P}(C_P)\leq\len_{R_P}(R_P)$.
Then $a_1\geq a_0$.
\end{lem}

\begin{proof}
Suppose that $a_1<a_0$, that is, that $a_1-a_0+1\leq 0$.
Set $K=\Ker(\partial)$ and consider the exact sequence
\begin{equation*} 
0\to K\to R^{a_1}\xra{\partial} R^{a_0}\to C\to 0.
\end{equation*}
Localize this sequence at an arbitrary $P\in\ass(R)$,
and count lengths to find that
$$0\leq\len_{R_P}(K_P)\leq(a_1-a_0+1)\len_{R_P}(R_P)\leq 0.$$
It follows that $K_P=0$ for all $P\in\ass(R)$.

Set $L=\im(\partial)$ and localize the exact sequence
$$0\to K \to R^{a_1}\xra\tau L\to 0$$
to conclude that $L_P\cong R_P^{a_1}$ for each $P\in\ass(R)$.
That is, the $R$-module $L$ has rank $a_1$.
Hence, the third step in the next sequence:
$$
a_1\geq\mu_R(L)
\geq\rank_R(L)=a_1.$$
The first step is from the surjection $\tau$.
It follows that $\mu_R(L)=\rank_R(L)$,
hence we conclude that $L$ is free; see, e.g., \cite[(1.12)]{sather:ecmfgd}.
The exact sequence
$$0\to L\to R^{a_0}\to C\to 0$$
implies that $\pd_R(C)$ is finite, a contradiction.
So, we have $a_1\geq a_0$, as desired.
\end{proof}

\begin{thm} \label{prop201}
Let 
$(R,\m)$ be a local ring such that $\ass(R)=\Min(R)$. Let 
$C$ be a semidualizing $R$-module such that $C\not\cong R$,
and consider an exact sequence
$$R^{a_1}\xra{\partial} R^{a_0}\to C\to 0.$$
For each $P\in\ass(R)$, assume that one of the following conditions holds:
\begin{enumerate}[\quad\rm(1)]
\item\label{prop201a}
$R_P$ is Gorenstein;
\item\label{prop201b}
$P^2R_P=0$; 
\item\label{prop201c}
$R_P\cong S/\n^e$ for some regular local ring $(S,\n)$ containing a field and some integer $e\geq 1$; or
\item\label{prop201d}
$R_P$ is isomorphic to a localization of a Cohen-Macaulay ring of the form 
$S/I$ where $S$ is a regular local ring containing a field with 
$x_1,\ldots,x_n\in S$ 
a regular system of parameters for $S$ such that $I$ is generated by monomials in the $x_i$.
\end{enumerate}
Then $\len_{R_P}(C_P)=\len_{R_P}(R_P)$
for each $P\in\ass(R)$. It follows that $a_1\geq a_0$
and that $e(J;C)=e(J;R)$ for each $\m$-primary ideal $J$.
\end{thm}

\begin{proof}
We first show that
$\len_{R_P}(C_P)=\len_{R_P}(R_P)$
for each $P\in\ass(R)$.
If $P$ satisfies condition~\eqref{prop201a} or~\eqref{prop201b}, this 
is a consequence of Fact~\ref{fact202}.
Under conditions~\eqref{prop201c} and~\eqref{prop201d}, 
we apply Lemma~\ref{lem202} and Corollary~\ref{lem204y}, respectively.

Now, the conclusion $a_1\geq a_0$ follows from
Lemma~\ref{lem201}, since Fact~\ref{fact202} implies that $\pd_R(C)=\infty$. The equality
$e(J;C)=e(J;R)$ for each $\m$-primary ideal $J$
follows from the additivity formula as in the proof of Theorem~\ref{thm201}.
\end{proof}

The next result shows how the existence of a non-trivial
semidualizing module yields an affirmative answer to~\cite[(2.6)]{jorgensen:gbscm}.

\begin{cor} \label{cor401}
Let 
$R$ be a Cohen-Macaulay local ring with a dualizing module $D$. Let 
$C$ be a semidualizing $R$-module such that $D\not\cong C\not\cong R$.
If for each $P\in\ass(R)$ one of the conditions~\eqref{prop201a}--\eqref{prop201d} 
from Theorem~\ref{prop201}
holds, then $\beta^R_1(D)\geq 2\beta^R_0(D)$.
\end{cor}

\begin{proof}
Set $\cd=\Hom CD$.
The condition $D\not\cong C$ implies that $\cd\not\cong R$ by
Fact~\ref{fact201}(f).
Hence, Theorem~\ref{prop201} implies that 
$\beta^R_1(\cd)\geq\beta^R_0(\cd)$
and
$\beta^R_1(C)\geq\beta^R_0(C)$.
This explains the second step in the next sequence:
\begin{align*}
\beta^R_1(D)
&=\beta^R_1(C)\beta^R_0(\cd)+\beta^R_0(C)\beta^R_1(\cd)
\geq 2\beta^R_0(C)\beta^R_0(\cd)
= 2\beta^R_0(D).
\end{align*}
The first and third steps follow from  Fact~\ref{fact201}(g).
\end{proof}

\section*{Acknowledgments}

We are grateful to C\u{a}t\u{a}lin Ciuperc\u{a} for helpful discussions about this material.

\providecommand{\bysame}{\leavevmode\hbox to3em{\hrulefill}\thinspace}
\providecommand{\MR}{\relax\ifhmode\unskip\space\fi MR }
\providecommand{\MRhref}[2]{%
  \href{http://www.ams.org/mathscinet-getitem?mr=#1}{#2}
}
\providecommand{\href}[2]{#2}

\end{document}